\documentclass[11pt]{article}
\usepackage[a4paper,margin=1in]{geometry}
\usepackage{amsmath,amssymb,amsthm,mathtools}
\usepackage{graphicx}
\usepackage{microtype}
\usepackage{enumitem}
\usepackage[hidelinks]{hyperref}

\newtheorem{theorem}{Theorem}[section]
\newtheorem{proposition}[theorem]{Proposition}
\newtheorem{lemma}[theorem]{Lemma}
\newtheorem{corollary}[theorem]{Corollary}
\newtheorem{remark}[theorem]{Remark}
\newcommand{\T}{\mathbb T}
\newcommand{\D}{\mathbb D}

\title{Radial convergence and divergence along monotone approaches in the infinite-dimensional polydisc}
\author{Jiawei Sun, Yufeng Lu  and Chao Zu\thanks{Corresponding author.}}
\date{}

\begin{document}
\maketitle

\begin{abstract}

We answer two questions of Aleman, Olsen, and Saksman concerning radial convergence for bounded analytic functions on the infinite-dimensional polydisc. First, we construct a boundary-point-dependent counterexample in $H^\infty(\mathbb{T}^\infty)$ for which every radius vector has non-increasing components and every fixed coordinate increases monotonically to $1$, yet the radial values converge almost everywhere to $0$ instead of the nonzero boundary value. We then construct a boundary-point-independent sequence of radius vectors with non-increasing components for which convergence to the boundary function still fails almost everywhere, thereby answering their second question as well. In contrast, if a boundary-point-independent approach is additionally componentwise increasing along the sequence, then the associated product Poisson operators satisfy an $L^p$ maximal inequality and converge almost everywhere and in $L^p$ for $1<p<\infty$. Consequently, the corresponding radial convergence holds for $H^p(\mathbb{T}^\infty)$ on the natural domains of the holomorphic extensions.

\end{abstract}

\section{Introduction}

Fatou's classical theorem asserts that bounded analytic functions on the unit disc have radial limits almost everywhere; see \cite{Fatou}. In finitely many variables, boundary convergence for product Poisson integrals is closely related to the strong differentiation theorem of Jessen, Marcinkiewicz, and Zygmund \cite{JMZ}. The infinite-dimensional setting is qualitatively different: convergence of each fixed coordinate to the boundary does not by itself control the infinitely many remaining coordinates, and the geometry of the radial approach becomes an essential part of the problem.

Aleman, Olsen, and Saksman developed a boundary theory for Hardy spaces on the infinite-dimensional polydisc \cite{AOS}. Among their results, they obtained positive convergence theorems for several restricted radial approaches and also showed that unrestricted convergence may fail dramatically. In particular, their Theorem 4 produces a zero-free function in  $H^\infty(\T^\infty)$ for which the correct boundary value is not recovered almost everywhere along suitable radial approaches. Their construction includes both a boundary-point-dependent approach, for which every fixed coordinate may be chosen to increase to $1$, and a boundary-point-independent one.

They subsequently asked whether the counterexample can be made compatible with an additional ordering of the coordinates. More precisely, Question 2 in \cite{AOS} asks whether radial convergence can still fail if every radius vector satisfies
\[
 r^{(s)}_1\ge r^{(s)}_2\ge\cdots,
\]
and Question 3 asks whether such a counterexample may in addition be chosen independently of the boundary point.

It is useful to distinguish this ordering of the coordinates of a fixed radius vector from the different requirement
\[
 r_n^{(s)}\le r_n^{(s+1)}\qquad(n,s\ge1),
\]
which imposes monotonicity of each fixed coordinate along the approach. The latter condition is not part of Questions 2 and 3, but it turns out to determine the positive side of the problem. We call the latter condition \emph{componentwise increasing} in $s$, and call an approach \emph{boundary-point independent} when the sequence of radius vectors is fixed independently of $e^{i\theta}\in\T^\infty$.

We answer both questions of Aleman, Olsen, and Saksman affirmatively. For Question 2, we construct a function \(f\in H^\infty(\mathbb T^\infty)\), with nonzero boundary values almost everywhere, and a boundary-point-dependent radial approach such that every radius vector has non-increasing components and every fixed coordinate increases monotonically to $1$, while
\[
f\bigl(r^{(s)}(\theta)e^{i\theta}\bigr)\longrightarrow0
\]
for almost every \(e^{i\theta}\in\mathbb T^\infty\), see Section~2. Thus the failure persists even under the additional monotonicity of each fixed coordinate along the sequence.

For Question 3, we discretize the radii occurring in this construction. The resulting sequence is independent of the boundary point, every radius vector again has non-increasing components, and every fixed coordinate tends to $1$. Nevertheless, for almost every boundary point the corresponding values of $f$ have a subsequence converging to $0$, and hence fail to converge to the nonzero boundary value, see Section 3.

The latter result shows that boundary-point independence alone does not restore Fatou convergence. This naturally raises the question of what happens if one also requires each fixed coordinate to increase monotonically along the approach. Under this additional hypothesis, the situation changes. We prove that if
\[
0\le r^{(1)}_n\le r^{(2)}_n\le\cdots\uparrow1
\qquad(n\ge1),
\]
and the sequence is independent of the boundary point, then for every $1<p<\infty$,
\[
 \left\|\sup_{s\ge1}|P_{r^{(s)}}g|\right\|_p
 \le \frac{p}{p-1}\|g\|_p,
 \qquad g\in L^p(\T^\infty).
\]
Consequently \(P_{r^{(s)}}g\to g\) almost everywhere and in \(L^p(\mathbb T^\infty)\), see Theorem \ref{thm:coordinatewise-monotone}. Applied to Hardy functions, this yields almost-everywhere radial convergence on the natural domains of their holomorphic extensions.

The counterexamples are obtained by adapting the block construction in Theorem 4 of \cite{AOS}. The main new ingredient is an explicit optimization of one-variable Poisson kernels. The maximizing radii yield a logarithmic gain while the contribution at the origin remains summable; suitable powers are then chosen so that the admissible radii are ordered from one coordinate to the next. A finite discretization of these radii gives the boundary-point-independent construction.

For the positive result, we use a convolution form of the Rota--Doob martingale argument. We establish a maximal theorem for a factorized family of convolution measures on a compact abelian group and apply it to the product Poisson measures on \(\mathbb T^\infty\). This gives both the maximal inequality and the almost-everywhere convergence stated above.


For convenience, we adopt the following notation throughout the paper. Let
\[
 \T^\infty=\prod_{n=1}^\infty\T,
 \qquad
 \D^\infty=\prod_{n=1}^\infty\D,
 \qquad
 m_\infty=\bigotimes_{n=1}^\infty m,
\]
where $m$ is normalized Haar measure on $\T$. We identify $\T$ with $\mathbb R/(2\pi\mathbb Z)$ and write $dm(t)=dt/(2\pi)$. Set $c_0=\{z=(z_n)_{n\ge1}\in\mathbb C^\infty:z_n\to0\}$ and $\ell^2=\{z=(z_n)_{n\ge1}:\sum_n|z_n|^2<\infty\}$. For $\theta=(\theta_n)$ and $r=(r_n)\in[0,1)^\infty$, we write $e^{i\theta}=(e^{i\theta_n})$ and $re^{i\theta}=(r_ne^{i\theta_n})$. For $1\le p\le\infty$, $H^p(\T^\infty)$ denotes the subspace of $L^p(\T^\infty)$ whose Fourier spectrum is contained in the finitely supported nonnegative multi-indices. For $1\le p<\infty$, the canonical holomorphic extension is defined on $\ell^2\cap\D^\infty$, whereas for $p=\infty$ the natural bounded holomorphic extension is defined on $c_0\cap\D^\infty$; see \cite[Chapter 5 \& Chapter 13]{DefantBook}. We use the same symbol for the boundary function and its holomorphic extension, and for $f\in H^\infty(\T^\infty)$ write $f^*$ for its boundary representative. The symbols $\|\cdot\|_p$, $\mathbb E$, and $\operatorname{Var}$ have their usual meanings with respect to $m_\infty$. The letters $c,C>0$ denote absolute constants whose values may change from line to line; $A\asymp B$ means $cB\le A\le CB$, and $A=O(B)$ means $|A|\le CB$.

\section{Question 2: a point-dependent monotone counterexample}

The following theorem is the main result of this section.

\begin{theorem}\label{thm:q2}
There exists $f\in H^\infty(\T^\infty)$ with $f^*\ne0$ almost everywhere and a sequence of Borel measurable maps $r^{(s)}:\T^\infty\to c_0\cap[0,1)^\infty$, $s\ge1$, such that, for $m_\infty$-almost every $e^{i\theta}\in\T^\infty$, the vectors $r^{(s)}(\theta)=\bigl(r_n^{(s)}(\theta)\bigr)_{n\ge1}$ have the following properties:
\begin{enumerate}[label=\rm(\arabic*)]
\item $r^{(s)}_1(\theta)\ge r^{(s)}_2(\theta)\ge\cdots$ for every $s$;
\item $r_n^{(s)}(\theta)\le r_n^{(s+1)}(\theta)$ for all $n,s$, and $r_n^{(s)}(\theta)\to1$ as $s\to\infty$ for every $n$;
\item $f\bigl(r^{(s)}(\theta)e^{i\theta}\bigr)\to0$ as $s\to\infty$.
\end{enumerate}
\end{theorem}

We adapt the counterexample constructed in Theorem~4 of \cite{AOS}. We divide the variables into finite blocks and construct analytic functions $U_{k,j}$ with nonnegative real parts $u_{k,j}$. For the $k$th block, write $z_{(k)}=(z_{k,1},\ldots,z_{k,L_k})$ and define the corresponding block function by $B_k\bigl(z_{(k)}\bigr):=\sum_{j=1}^{L_k}u_{k,j}(z_{k,j})$.
The counterexample function is the infinite product
\[
 f(z)=\prod_{k=1}^{\infty}\prod_{j=1}^{L_k}
 \exp\bigl(-U_{k,j}(z_{k,j})\bigr),
 \qquad z\in c_0\cap\D^\infty.
\]
We choose the parameters so that the contributions at the origin are summable, while the selected block radii satisfy $B_k\bigl(r_{(k)}(\theta)e^{i\theta_{(k)}}\bigr)\to\infty ~(k\to \infty)$ for almost every boundary point. The first property ensures that the product is nonzero, whereas the second forces its modulus to tend to zero along the resulting radial approach.

\subsection{Construction of \texorpdfstring{$u_{k,j}$}{u(k,j)} and the associated radii}

For each $k\ge1$, choose a positive integer $L_k$ and a number $0<\eta_k<1/16$. The integer $L_k$ is the length of the $k$th block, while $\eta_k$ is the scale parameter used for every factor in that block. Partition the coordinates into consecutive blocks
\[
 \{1,\ldots,L_1\},\quad
 \{L_1+1,\ldots,L_1+L_2\},\quad
 \ldots .
\]
For $1\le j\le L_k$, write $z_{k,j}$ and $\theta_{k,j}$ for the $j$th coordinate and angle in the $k$th block. For $0<a<1$ define
\[
 P_a(z)=\Re\frac{1+az}{1-az}
 =\frac{1-a^2|z|^2}{|1-az|^2}.
\]
For integers $q_{k,j}\ge1$, set $U_{k,j}(z)=\eta_k\frac{1+(1-\eta_k)z^{q_{k,j}}}{1-(1-\eta_k)z^{q_{k,j}}}$ and $u_{k,j}=\Re U_{k,j}=\eta_kP_{1-\eta_k}(z^{q_{k,j}})$. Then $u_{k,j}$ is nonnegative and harmonic and $u_{k,j}(0)=\eta_k$. For $z_{(k)}=(z_{k,1},\ldots,z_{k,L_k})$, define $B_k\bigl(z_{(k)}\bigr):=\sum_{j=1}^{L_k}u_{k,j}(z_{k,j})$. We shall choose the parameters so that $\sum_{k=1}^\infty L_k\eta_k<\infty$. Under this condition, define $f_{k,j}(w)=\exp\bigl(-U_{k,j}(w)\bigr)$ and $F_N(z)=\prod_{k=1}^{N}\prod_{j=1}^{L_k}f_{k,j}(z_{k,j})$. Since $\Re U_{k,j}\ge0$, one has $\|F_N\|_\infty\le1$. Fix $0<\rho<1$. If $z\in c_0$ and $\|z\|_\infty\le\rho$, then, since $q_{k,j}\ge1$, $|U_{k,j}(z_{k,j})|\le\eta_k(1+\rho)/(1-\rho)$.
Hence
\[
 \sum_{k=1}^{\infty}\sum_{j=1}^{L_k}
 \sup_{\|z\|_\infty\le\rho}|U_{k,j}(z_{k,j})|
 \le
 \frac{1+\rho}{1-\rho}
 \sum_{k=1}^{\infty}L_k\eta_k
 <\infty.
\]
Therefore the series $\sum_{k=1}^{\infty}\sum_{j=1}^{L_k}U_{k,j}(z_{k,j})$ converges uniformly on every smaller closed ball of $c_0$, and hence compact-uniformly on $c_0\cap\D^\infty$. It follows that
\[
 f(z):=\prod_{k=1}^{\infty}\prod_{j=1}^{L_k}f_{k,j}(z_{k,j})
 =\exp\!\left(-\sum_{k=1}^{\infty}\sum_{j=1}^{L_k}U_{k,j}(z_{k,j})\right)
\]
is a bounded holomorphic function on $c_0\cap\D^\infty$, and $F_N\to f$ compact-uniformly there. Moreover, $|f(z)|\le1$ and $f(0)=\exp\!\left(-\sum_{k=1}^{\infty}L_k\eta_k\right)>0$. By the standard isometric identification between bounded holomorphic functions on $c_0\cap\D^\infty$ and $H^\infty(\T^\infty)$ (see \cite[Chapter 5]{DefantBook}), we regard $f$ as an element of $H^\infty(\T^\infty)$. By Corollary~2 of \cite{AOS},
\begin{equation}\label{eq:boundary-nonzero}
 f^*(e^{i\theta})\ne0
 \quad\text{for }m_\infty\text{-almost every }e^{i\theta}\in\T^\infty.
\end{equation}
Taking real parts in the absolutely convergent series gives, for every $z\in c_0\cap\D^\infty$,
\begin{equation}\label{eq:modulus}
 |f(z)|
 =\exp\!\left(-\sum_{k=1}^{\infty}B_k\bigl(z_{(k)}\bigr)\right).
\end{equation}
The remaining restrictions on $\eta_k$ and $L_k$ will be imposed by the block estimates.

\begin{proposition}\label{prop:max}
For $0<\varphi<1/4$, the function $a\mapsto P_a(e^{i\varphi})$ attains its maximum at $a^*(\varphi)=\cos\varphi/(1+\sin\varphi)$, and $P_{a^*(\varphi)}(e^{i\varphi})=1/\sin\varphi$.
\end{proposition}

\begin{proof}
Differentiation gives
\[
 \frac{\partial}{\partial a}P_a(e^{i\varphi})
 =\frac{2((1+a^2)\cos\varphi-2a)}{(1-2a\cos\varphi+a^2)^2}.
\]
The critical-point equation $(1+a^2)\cos\varphi=2a$ has the roots $a=(1\pm\sin\varphi)/\cos\varphi$.
Only $a^*(\varphi)=\cos\varphi/(1+\sin\varphi)$ belongs to $(0,1)$, and the derivative changes sign from positive to negative there. Substitution gives the stated value.
\end{proof}

Fix $k$. With $\operatorname{dist}(x,2\pi\mathbb Z)$ denoting the Euclidean distance from $x\in\mathbb R$ to $2\pi\mathbb Z$, put $\varphi_{k,j}(t)=\operatorname{dist}(q_{k,j}t,2\pi\mathbb Z)$. For fixed $t$, the Poisson-kernel parameter in $u_{k,j}(re^{it})$ is $a=(1-\eta_k)r^{q_{k,j}}$. Hence, whenever $2\eta_k<\varphi_{k,j}(t)<1/4$, Proposition~\ref{prop:max} shows that $u_{k,j}(re^{it})$ is maximized at the radius
\begin{equation}\label{eq:maximizing-radius}
 r=
 \left(
 \frac{a^*(\varphi_{k,j}(t))}{1-\eta_k}
 \right)^{1/q_{k,j}}.
\end{equation}
We choose the integers $q_{k,j}$ so that these maximizing radii are ordered from one coordinate to the next.

Set $\alpha_k=a^*(1/4)/(1-\eta_k)$ and $\beta_k=a^*(2\eta_k)/(1-\eta_k)$. The function $\varphi\mapsto a^*(\varphi)/(1-\eta_k)$ is strictly decreasing on $(0,\pi/2)$. Since $0<2\eta_k<1/4$ and $\eta_k<1/16$, one has $0<\alpha_k<\beta_k<1$, and every maximizing radius in \eqref{eq:maximizing-radius} belongs to $[\alpha_k^{1/q_{k,j}},\beta_k^{1/q_{k,j}}]$. Thus it is enough to arrange that these intervals are strictly ordered, namely $\alpha_k^{1/q_{k,j}}>\beta_k^{1/q_{k,j+1}}$ for $1\le j<L_k$. Since $\log\alpha_k<\log\beta_k<0$, this is equivalent to $q_{k,j}>(\log\alpha_k/\log\beta_k)q_{k,j+1}$. We therefore choose $q_{k,L_k}=1$ and then choose $q_{k,j}$ recursively backwards so that this inequality holds. Such a choice is possible because each step imposes only a finite lower bound.

Finally, define
\begin{equation}\label{eq:rkj}
 r_{k,j}(t)=
 \begin{cases}
 \displaystyle
 \left(
 \frac{a^*(\varphi_{k,j}(t))}{1-\eta_k}
 \right)^{1/q_{k,j}},
 &2\eta_k<\varphi_{k,j}(t)<1/4,\\[3mm]
 (\alpha_k\beta_k)^{1/(2q_{k,j})},
 &\text{otherwise}.
 \end{cases}
\end{equation}
In either case, $\alpha_k^{1/q_{k,j}}\le r_{k,j}(t)\le\beta_k^{1/q_{k,j}}$. Together with the preceding separation condition, this gives
\begin{equation}\label{eq:block-monotone}
 r_{k,j}(t)>r_{k,j+1}(s)
 \qquad(t,s\in\T,\ 1\le j<L_k).
\end{equation}
On the set $2\eta_k<\varphi_{k,j}(t)<1/4$, Proposition~\ref{prop:max} gives
\begin{equation}\label{eq:peak-value}
 u_{k,j}\!\left(r_{k,j}(t)e^{it}\right)
 =\frac{\eta_k}{\sin\varphi_{k,j}(t)}.
\end{equation}
Outside this set,
\[
 (1-\eta_k)r_{k,j}(t)^{q_{k,j}}
 =\sqrt{a^*(1/4)a^*(2\eta_k)}
 \le \sqrt{a^*(1/4)}<1.
\]
Thus the parameter of the Poisson kernel remains in the fixed interval $[0,\sqrt{a^*(1/4)}]$. Since $P_a(e^{i\varphi})\le(1+a)/(1-a)$ for $0\le a<1$, the corresponding Poisson kernel is bounded by an absolute constant. Therefore, $u_{k,j}\!\left(r_{k,j}(t)e^{it}\right)\le C\eta_k$.

Geometric illustrations of the maximizing radii and the piecewise-defined radii appear in Appendix~\ref{app:figures}.

\subsection{Block estimates and parameter choice}

For a real-valued $X\in L^2(\T^\infty,m_\infty)$, define $\mathbb E X:=\int_{\T^\infty}X\,dm_\infty$ and $\operatorname{Var}(X):=\mathbb E[(X-\mathbb E X)^2]=\mathbb E(X^2)-(\mathbb E X)^2$. For the selected radii, write $r_{(k)}(\theta)e^{i\theta_{(k)}}:=(r_{k,j}(\theta_{k,j})e^{i\theta_{k,j}})_{j=1}^{L_k}$. Then
\[
 B_k\bigl(r_{(k)}(\theta)e^{i\theta_{(k)}}\bigr)
 =\sum_{j=1}^{L_k}
 u_{k,j}\!\left(r_{k,j}(\theta_{k,j})e^{i\theta_{k,j}}\right).
\]

\begin{proposition}\label{prop:mean}
There exist absolute constants $c,C>0$ such that
\[
 cL_k\eta_k\log\frac1{\eta_k}
 \le \mathbb E\,B_k\bigl(r_{(k)}(\theta)e^{i\theta_{(k)}}\bigr)
 \le C L_k\eta_k\log\frac1{\eta_k}.
\]
\end{proposition}

\begin{proof}
The map $t\mapsto q_{k,j}t$ preserves normalized Haar measure. More explicitly, for every $2\pi$-periodic integrable function $F$,
\[
 \int_0^{2\pi}F(q_{k,j}t)\,\frac{dt}{2\pi}
 =\frac1{q_{k,j}}\int_0^{2\pi q_{k,j}}F(s)\,\frac{ds}{2\pi}
 =\int_0^{2\pi}F(s)\,\frac{ds}{2\pi}.
\]
Therefore,
\[
 \int_{\T}
 u_{k,j}\!\left(r_{k,j}(t)e^{it}\right)\,dm(t)
 =\frac{\eta_k}{\pi}
 \int_{2\eta_k}^{1/4}\frac{d\varphi}{\sin\varphi}
 +O(\eta_k).
\]
Since $\sin\varphi\asymp\varphi$ on $(0,1/4)$, this is comparable to $\eta_k\log(1/\eta_k)$. Summing over $j$ proves the estimate.
\end{proof}

\begin{proposition}\label{prop:variance}
There exists an absolute constant $C>0$ such that
\[
 \operatorname{Var}\!\left(B_k\bigl(r_{(k)}(\theta)e^{i\theta_{(k)}}\bigr)\right)
 \le C L_k\eta_k.
\]
\end{proposition}

\begin{proof}
Regard $u_{k,j}$ as the function $e^{i\theta}\mapsto u_{k,j}(r_{k,j}(\theta_{k,j})e^{i\theta_{k,j}})$ on $\T^\infty$. Since $u_{k,i}$ and $u_{k,j}$ depend on distinct coordinates whenever $i\ne j$, Fubini's theorem gives $\mathbb E(u_{k,i}u_{k,j})=\mathbb E u_{k,i}\,\mathbb E u_{k,j}$.
Thus the cross terms vanish and
\[
 \operatorname{Var}\!\left(B_k\bigl(r_{(k)}(\theta)e^{i\theta_{(k)}}\bigr)\right)
 =\sum_{j=1}^{L_k}\operatorname{Var}(u_{k,j})
 \le\sum_{j=1}^{L_k}
 \int_{\T}u_{k,j}\!\left(r_{k,j}(t)e^{it}\right)^2\,dm(t).
\]
To estimate the integral over the whole circle, use \eqref{eq:peak-value} for $t\in\T$ satisfying $2\eta_k<\varphi_{k,j}(t)<1/4$, while for the remaining $t$ one has $u_{k,j}(r_{k,j}(t)e^{it})=O(\eta_k)$. Hence
\[
 \int_{\T}u_{k,j}\!\left(r_{k,j}(t)e^{it}\right)^2\,dm(t)
 \le C\eta_k^2\int_{2\eta_k}^{1/4}\frac{d\varphi}{\varphi^2}
 +C\eta_k^2
 \le C\eta_k.
\]
Therefore,
\[
 \operatorname{Var}\!\left(B_k\bigl(r_{(k)}(\theta)e^{i\theta_{(k)}}\bigr)\right)
 \le C L_k\eta_k.
\]
\end{proof}

By Chebyshev's inequality and Propositions~\ref{prop:mean}--\ref{prop:variance},
\[
 m_\infty\!\left(
 B_k\bigl(r_{(k)}(\theta)e^{i\theta_{(k)}}\bigr)
 <\frac12\mathbb E B_k\bigl(r_{(k)}(\theta)e^{i\theta_{(k)}}\bigr)
 \right)
 \le
 \frac{C}{L_k\eta_k(\log(1/\eta_k))^2}.
\]
To ensure that the block means tend to infinity and that the exceptional measures are summable, it is enough to require $L_k\eta_k\log(1/\eta_k)\to\infty$ and $\sum_k(L_k\eta_k\log^2(1/\eta_k))^{-1}<\infty$. These requirements are compatible with the summability condition above. For each $k\ge1$, take $\eta_k=e^{-(k+2)^3}$ and $L_k=\left\lfloor1/(k^2\eta_k)\right\rfloor$. Since $1/(k^2\eta_k)\ge2$, one has $1/(2k^2)\le L_k\eta_k\le1/k^2$. Hence the summability condition holds, while Propositions~\ref{prop:mean}--\ref{prop:variance} give $\mathbb E B_k(r_{(k)}(\theta)e^{i\theta_{(k)}})\to\infty$ and
\[
 \sum_{k=1}^\infty
 m_\infty\!\left(
 B_k\bigl(r_{(k)}(\theta)e^{i\theta_{(k)}}\bigr)
 <\frac12\mathbb E B_k\bigl(r_{(k)}(\theta)e^{i\theta_{(k)}}\bigr)
 \right)<\infty.
\]
By the Borel--Cantelli lemma, see \cite[Vol.~1, Chapter~I, Section~I]{LiQueffelec}, Proposition~\ref{prop:mean} therefore yields
\begin{equation}\label{eq:B-ae}
 B_k\bigl(r_{(k)}(\theta)e^{i\theta_{(k)}}\bigr)\longrightarrow\infty
 \quad\text{for }m_\infty\text{-almost every }e^{i\theta}\in\T^\infty.
\end{equation}

\subsection{Construction of the monotone radial approach}

Choose $\varepsilon_1>0$ such that $\varepsilon_1<\min\{1,1-\beta_1^{1/q_{1,1}}\}$, and recursively choose
\[
 0<\varepsilon_s<\min\left\{
 \frac{\varepsilon_{s-1}}2,\frac1s,
 1-\max_{1\le k\le s}\beta_k^{1/q_{k,1}}
 \right\},
 \qquad s\ge2.
\]
Then $\varepsilon_s\downarrow0$ and $1-\varepsilon_s>\max_{1\le k\le s}\beta_k^{1/q_{k,1}}$.
Define
\[
 r^{(s)}_{k,j}(\theta)=
 \begin{cases}
 1-\varepsilon_s,&k<s,\\
 r_{s,j}(\theta_{s,j}),&k=s,\\
 0,&k>s.
 \end{cases}
\]
Equation \eqref{eq:block-monotone} and the choice of $\varepsilon_s$ show that the components of $r^{(s)}(\theta)$ are non-increasing. For fixed $(k,j)$, its values are
\[
 0,\ldots,0,\ r_{k,j}(\theta_{k,j}),\
 1-\varepsilon_{k+1},1-\varepsilon_{k+2},\ldots,
\]
and hence increase to $1$. Each $r^{(s)}$ is a Borel map with finite support and takes values in $c_0\cap[0,1)^\infty$.

Finally, \eqref{eq:modulus} and \eqref{eq:B-ae} give
\[
 \left|f\bigl(r^{(s)}(\theta)e^{i\theta}\bigr)\right|
 \le
 \exp\!\left(-B_s\bigl(r_s(\theta)e^{i\theta_s}\bigr)\right)
 \longrightarrow0
\]
for $m_\infty$-almost every $e^{i\theta}\in\T^\infty$. By \eqref{eq:boundary-nonzero}, $f^*(e^{i\theta})\ne0$ almost everywhere. This proves Theorem~\ref{thm:q2}.

\section{Question 3: boundary-point-independent approaches}

Throughout this section, we retain the function $f$ and the notation from Section~2.

\subsection{A boundary-point-independent counterexample}

\begin{lemma}\label{lem:finite-grid}
Fix $k\ge1$ and $1\le j\le L_k$. There is a finite set $D_{k,j}\subset[\alpha_k^{1/q_{k,j}},\beta_k^{1/q_{k,j}}]$ such that, for every $t\in\mathbb T$,
\[
 \max_{\rho\in D_{k,j}}u_{k,j}(\rho e^{it})
 \ge \frac12\,u_{k,j}\!\left(r_{k,j}(t)e^{it}\right).
\]
\end{lemma}

\begin{proof}
Set $I_{k,j}=[\alpha_k^{1/q_{k,j}},\beta_k^{1/q_{k,j}}]$ and $\rho^0_{k,j}=(\alpha_k\beta_k)^{1/(2q_{k,j})}$. For $2\eta_k\le\varphi\le1/4$, Proposition~\ref{prop:max} gives the maximizing radius $\rho_\varphi=\bigl(a^*(\varphi)/(1-\eta_k)\bigr)^{1/q_{k,j}}\in I_{k,j}$. The map $(\rho,\psi)\mapsto \eta_kP_{(1-\eta_k)\rho^{q_{k,j}}}(e^{i\psi})$ is continuous on the compact set $I_{k,j}\times[2\eta_k,1/4]$, and the map $\varphi\mapsto\rho_\varphi$ is continuous on $[2\eta_k,1/4]$. Hence, for each fixed $\varphi\in[2\eta_k,1/4]$, the function $\psi\mapsto P_{(1-\eta_k)\rho_\varphi^{q_{k,j}}}(e^{i\psi})/P_{(1-\eta_k)\rho_\psi^{q_{k,j}}}(e^{i\psi})$ is continuous in a neighborhood of $\varphi$ and takes the value $1$ at $\psi=\varphi$. Therefore there exists a neighborhood $V_\varphi$ of $\varphi$ such that
\[
 \eta_kP_{(1-\eta_k)\rho_\varphi^{q_{k,j}}}(e^{i\psi})
 \ge \frac12\,
 \eta_kP_{(1-\eta_k)\rho_\psi^{q_{k,j}}}(e^{i\psi}),
 \qquad \psi\in V_\varphi.
\]
By compactness of $[2\eta_k,1/4]$, finitely many of these neighborhoods, say $V_{\varphi_1},\ldots,V_{\varphi_m}$, cover the interval. Put $D^*_{k,j}=\{\rho_{\varphi_1},\ldots,\rho_{\varphi_m}\}$ and $D_{k,j}=D^*_{k,j}\cup\{\rho^0_{k,j}\}$.
If $2\eta_k<\varphi_{k,j}(t)<1/4$, then \eqref{eq:rkj} gives $r_{k,j}(t)=\rho_{\varphi_{k,j}(t)}$, and the preceding finite-cover estimate yields
\[
 \max_{\rho\in D_{k,j}}u_{k,j}(\rho e^{it})
 \ge \frac12\,u_{k,j}\!\left(r_{k,j}(t)e^{it}\right).
\]
If $\varphi_{k,j}(t)\notin(2\eta_k,1/4)$, then \eqref{eq:rkj} gives $r_{k,j}(t)=\rho^0_{k,j}\in D_{k,j}$, so the same inequality holds with constant $1$.
\end{proof}

For each $k$, define $D_k=D_{k,1}\times\cdots\times D_{k,L_k}$. By the separation relation $\alpha_k^{1/q_{k,j}}>\beta_k^{1/q_{k,j+1}}$ established above, every $\rho=(\rho_1,\ldots,\rho_{L_k})\in D_k$ satisfies $\rho_1>\cdots>\rho_{L_k}$.

The next theorem is the main result of this subsection.

\begin{theorem}\label{thm:q3}
There exist $f\in H^\infty(\T^\infty)$ and a boundary-point-independent sequence $R^{(s)}\in c_0\cap[0,1)^\infty$ such that
\begin{enumerate}[label=\rm(\arabic*)]
\item $R^{(s)}_1\ge R^{(s)}_2\ge\cdots$ for every $s$;
\item $R^{(s)}_n\to1$ for every fixed $n$;
\item for almost every $e^{i\theta}\in\T^\infty$, the sequence $f(R^{(s)}e^{i\theta})$ has a subsequence converging to $0$, and therefore does not converge to $f^*(e^{i\theta})$.
\end{enumerate}
In general, for a fixed $n$, the sequence $R^{(s)}_n$ is not monotone in $s$.
\end{theorem}

\begin{proof}
Take the function $f$ constructed in Section~2. Choose $\varepsilon_k\downarrow0$ so that $1-\varepsilon_k>\beta_k^{1/q_{k,1}}$, and write $D_k=\{\rho^{k,1},\ldots,\rho^{k,N_k}\}$. For $1\le\ell\le N_k$, extend $\rho^{k,\ell}$ to $R^{k,\ell}\in c_0\cap[0,1)^\infty$ by
\[
 R^{k,\ell}_{m,j}=
 \begin{cases}
 1-\varepsilon_k,&m<k,\\
 \rho^{k,\ell}_j,&m=k,\\
 0,&m>k,
 \end{cases}
 \qquad 1\le j\le L_m.
\]
Equivalently,
\[
 R^{k,\ell}=
 \bigl(
 \underbrace{1-\varepsilon_k,\ldots,1-\varepsilon_k}_{L_1+\cdots+L_{k-1}\ \mathrm{coordinates}},
 \rho^{k,\ell}_1,\ldots,\rho^{k,\ell}_{L_k},
 0,0,\ldots
 \bigr).
\]
Let $(R^{(s)})_{s\ge1}$ be the sequence obtained by concatenating, in increasing order of $k$, the finite lists $R^{k,1},\ldots,R^{k,N_k}$. Equivalently, the sequence has the block form
\[
\begin{aligned}
&\bigl(\rho^{1,1}_1,\ldots,\rho^{1,1}_{L_1},0,0,\ldots\bigr),\ldots,
 \bigl(\rho^{1,N_1}_1,\ldots,\rho^{1,N_1}_{L_1},0,0,\ldots\bigr),\\
&\bigl(\underbrace{1-\varepsilon_2,\ldots,1-\varepsilon_2}_{L_1},
 \rho^{2,1}_1,\ldots,\rho^{2,1}_{L_2},0,0,\ldots\bigr),\ldots,\\
&\hspace{35mm}\ldots,\\
&\bigl(\underbrace{1-\varepsilon_k,\ldots,1-\varepsilon_k}_{L_1+\cdots+L_{k-1}},
 \rho^{k,\ell}_1,\ldots,\rho^{k,\ell}_{L_k},0,0,\ldots\bigr),\ldots .
\end{aligned}
\]

For every $k$ and $\rho\in D_k$, one has $1-\varepsilon_k>\rho_1>\cdots>\rho_{L_k}>0$, so every $R^{k,\ell}$ has non-increasing components. This proves~\rm(i). Fix a coordinate in the $k_0$th block. Every vector belonging to a later block $k>k_0$ has value $1-\varepsilon_k$ at this coordinate. Since only finitely many terms of $(R^{(s)})$ come from blocks $k\le k_0$ and $\varepsilon_k\downarrow0$, assertion~\rm(ii) follows.

Fix $\theta$ for which \eqref{eq:B-ae} holds. For each $k$ and $1\le j\le L_k$, choose $\rho_{k,j}(\theta)\in D_{k,j}$ at which $\max_{\rho\in D_{k,j}}u_{k,j}(\rho e^{i\theta_{k,j}})$ is attained, and put $\rho_k(\theta):=(\rho_{k,1}(\theta),\ldots,\rho_{k,L_k}(\theta))\in D_k$. By Lemma~\ref{lem:finite-grid}, this choice makes the block sum at least $\frac12 B_k\bigl(r_{(k)}(\theta)e^{i\theta_{(k)}}\bigr)$. Hence, by \eqref{eq:B-ae}, the block sum tends to infinity.

Choose $\ell(k,\theta)$ so that $\rho^{k,\ell(k,\theta)}=\rho_k(\theta)$. Then $(R^{k,\ell(k,\theta)})_{k\ge1}$ is a subsequence of $(R^{(s)})$, and \eqref{eq:modulus} gives
\[
 \left|f\bigl(R^{k,\ell(k,\theta)}e^{i\theta}\bigr)\right|
 \le
 \exp\!\left[-\sum_{j=1}^{L_k}u_{k,j}(\rho_{k,j}(\theta)e^{i\theta_{k,j}})\right]
 \longrightarrow0.
\]
Hence $0$ is a cluster point of the sequence. Since $f^*(e^{i\theta})\ne0$ almost everywhere by \eqref{eq:boundary-nonzero}, the full sequence cannot converge to the correct boundary value.
\end{proof}

\begin{remark}\label{rem:q3-wording}
Theorem~\ref{thm:q3} answers Question~3 as stated in \cite{AOS}. The condition required there is that the components of each radius vector are non-increasing. The sequence above need not satisfy the additional componentwise monotonicity $R_n^{(s)}\le R_n^{(s+1)}$ for every fixed $n$.
\end{remark}

\subsection{The additional componentwise increasing requirement}

We now impose the additional requirement that the sequence of radius vectors be componentwise increasing. We derive the resulting maximal estimate from a martingale representation of convolution operators. We first formulate the argument on a compact abelian group and then specialize it to the product Poisson measures on $\T^\infty$.

All $\sigma$-algebras in this subsection are understood relative to the ambient probability space. If $X_1,\ldots,X_n$ are measurable maps, then $\sigma(X_1,\ldots,X_n)$ denotes the $\sigma$-algebra generated by these maps. On $\T^\infty$, we write $z_n(\zeta)=\zeta_n$ for the $n$th coordinate map, where $\zeta=(\zeta_m)_{m\ge1}\in\T^\infty$.

If $(\Omega,\mathcal F,\mu)$ is a probability space, $F\in L^1(\Omega,\mathcal F,\mu)$, and $\mathcal G\subset\mathcal F$ is a sub-$\sigma$-algebra, then $\mathbb E_\mu[F\mid\mathcal G]$ denotes the $\mathcal G$-measurable function satisfying
\[
 \int_A\mathbb E_\mu[F\mid\mathcal G] \,d\mu
 =\int_A F\,d\mu
 \qquad(A\in\mathcal G).
\]
A sequence $(M_s,\mathcal F_s)$ is a martingale on $(\Omega,\mathcal F,\mu)$ if $(\mathcal F_s)$ is an increasing filtration, each $M_s$ is $\mathcal F_s$-measurable and integrable, and $\mathbb E_\mu[M_t\mid\mathcal F_s]=M_s$ for $s\le t$.

We begin with a standard martingale example. Let $H_n:=\prod_{j=n+1}^\infty\T_j$, viewed as a closed subgroup of $\T^\infty$, and let $\mu_n$ be its normalized Haar measure. Then $f*\mu_n=\mathbb E(f\mid\sigma(z_1,\ldots,z_n))$ and $\mu_n*\mu_m=\mu_{\min\{n,m\}}$. Consequently, for every fixed $N$, $(f*\mu_s*\mu_{s+1}*\cdots*\mu_N)_{s=1}^N$ is a martingale with respect to $\mathcal F_s=\sigma(z_1,\ldots,z_s)$. The same argument gives the following convolution formulation.

For a probability measure $\mu$ on a compact abelian group $G$, define $f*\mu$ by $(f*\mu)(x):=\int_G f(xy^{-1})\,d\mu(y)$. Let $\check\mu$ denote the image of $\mu$ under inversion.

\begin{proposition}\label{prop:compact-group-martingale}
Let $G$ be a compact abelian group with normalized Haar measure $m_G$, let $\mu_1,\ldots,\mu_N$ be probability measures on $G$, and put $d\mathbb P=dm_G(x)\,d\mu_1(u_1)\cdots d\mu_N(u_N)$ on $G^{N+1}$. With $F=f(xu_1^{-1}\cdots u_N^{-1})$ and $\mathcal F_s=\sigma(x,u_1,\ldots,u_{s-1})$, one has
\begin{equation}\label{eq:compact-dilation}
 \mathbb E_{\mathbb P}[F\mid\mathcal F_s]
 =
 (f*\mu_s*\cdots*\mu_N)(xu_1^{-1}\cdots u_{s-1}^{-1}).
\end{equation}
Consequently these conditional expectations form a martingale.
\end{proposition}

\begin{proof}
Keeping $x,u_1,\ldots,u_{s-1}$ fixed and integrating over $u_s,\ldots,u_N$ gives the right-hand side of \eqref{eq:compact-dilation}. The defining identity for conditional expectation follows first on measurable rectangles and then on $\mathcal F_s$.
\end{proof}

This martingale representation is the form of the Rota--Doob argument used below; compare Rota's dilation method \cite{Rota} and Doob's martingale framework \cite{DoobRatio}.

\begin{theorem}[Doob's maximal inequality; see Theorem~14.11 in \cite{Williams}]\label{thm:doob}
Let $(M_s,\mathcal F_s)_{s=1}^{N}$ be an $L^p$ martingale on a probability space, where $1<p<\infty$. Then
\[
 \left\|\max_{1\le s\le N}|M_s|\right\|_p
 \le\frac{p}{p-1}\|M_N\|_p.
\]
\end{theorem}

\begin{theorem}\label{thm:compact-group-maximal}
Let $G$ be a compact abelian group with normalized Haar measure $m_G$, and let $(\lambda_s)_{s\ge1}$ be probability measures on $G$ such that, for every $s\ge1$, there exists a probability measure $\kappa_s$ satisfying
\begin{equation}\label{eq:compact-chain}
 \lambda_s=\lambda_{s+1}*\kappa_s.
\end{equation}
Put $\mu_s:=\lambda_s*\check\lambda_s$ and $T_sf:=f*\mu_s$. Then, for every $1<p<\infty$ and every $f\in L^p(G)$,
\begin{equation}\label{eq:compact-maximal}
 \left\|\sup_{s\ge1}|T_sf|\right\|_{L^p(G)}
 \le \frac{p}{p-1}\|f\|_{L^p(G)}.
\end{equation}
If, in addition,
\begin{equation}\label{eq:compact-approx-id}
 \mu_s\overset{w^*}{\longrightarrow}\delta_e
 \qquad(s\to\infty),
\end{equation}
where $e$ is the identity of $G$, then
\[
 T_sf\longrightarrow f
 \quad\text{almost everywhere and in }L^p(G).
\]
\end{theorem}

\begin{proof}
Fix $N$ and set $\kappa_N:=\lambda_N$. Iterating \eqref{eq:compact-chain} gives $\kappa_s*\cdots*\kappa_N=\lambda_s$ for $1\le s\le N$. Equip $\Omega=G^{N+1}$ with $d\mathbb P=dm_G(z)\,d\kappa_1(u_1)\cdots d\kappa_N(u_N)$, and put $x=zu_N\cdots u_1$. For fixed $u_1,\ldots,u_N$, the map $z\mapsto x$ is a translation of $G$; hence Haar invariance shows that $(z,u_1,\ldots,u_N)\mapsto(x,u_1,\ldots,u_N)$ preserves $\mathbb P$. Under this measure-preserving change of variables, Proposition~\ref{prop:compact-group-martingale}, applied with $\mu_j=\kappa_j$, yields, after returning to $(z,u_1,\ldots,u_N)$, the $L^p$ martingale $M_s=(f*\lambda_s)(zu_N\cdots u_s)$, $1\le s\le N$.

Conditioning with respect to $\sigma(z)$ and using $\kappa_s*\cdots*\kappa_N=\lambda_s$ gives $\mathbb E_{\mathbb P}[M_s\mid\sigma(z)]=\int_G (f*\lambda_s)(zy)\,d\lambda_s(y)$. By the definition of convolution, the last expression equals $(f*\lambda_s)*\check\lambda_s(z)=T_sf(z)$. Therefore
\[
 \max_{1\le s\le N}|T_sf|
 \le
 \mathbb E_{\mathbb P}\!\left[
 \max_{1\le s\le N}|M_s|\,\middle|\,\sigma(z)
 \right].
\]
By the $L^p$-contractivity of conditional expectation and Doob's maximal inequality,
\[
 \left\|\max_{1\le s\le N}|T_sf|\right\|_{L^p(G)}
 \le
 \left\|\max_{1\le s\le N}|M_s|\right\|_{L^p(\Omega)}
 \le
 \frac{p}{p-1}\|M_N\|_{L^p(\Omega)}
 \le
 \frac{p}{p-1}\|f\|_{L^p(G)}.
\]
The last inequality follows from Haar invariance and the $L^p$-contractivity of convolution by the probability measure $\lambda_N$.
Letting $N\to\infty$ and using monotone convergence proves \eqref{eq:compact-maximal}.

Assume in addition that \eqref{eq:compact-approx-id} holds. If $\gamma$ is a character of $G$, then
\[
 T_s\gamma(x)
 =\gamma(x)\int_G\gamma(y^{-1})\,d\mu_s(y)
 \longrightarrow \gamma(x)
\]
uniformly on $G$, by the weak-* convergence $\mu_s\to\delta_e$. Hence $T_sQ\to Q$ uniformly for every trigonometric polynomial $Q$ on $G$.

Let $f\in L^p(G)$ and $\varepsilon>0$. For every trigonometric polynomial $Q$,
\[
 \limsup_{s\to\infty}|T_sf-f|
 \le
 \sup_{s\ge1}|T_s(f-Q)|+|f-Q|
\]
almost everywhere. Consequently, by \eqref{eq:compact-maximal} and Chebyshev's inequality,
\[
\begin{aligned}
 m_G\!\left(\limsup_{s\to\infty}|T_sf-f|>2\varepsilon\right)
 &\le
 m_G\!\left(\sup_{s\ge1}|T_s(f-Q)|>\varepsilon\right)
 +m_G(|f-Q|>\varepsilon)\\
 &\le
 \varepsilon^{-p}
 \left[\left(\frac{p}{p-1}\right)^p+1\right]
 \|f-Q\|_p^p.
\end{aligned}
\]
Since trigonometric polynomials are dense in $L^p(G)$, the right-hand side can be made arbitrarily small. Thus $T_sf\to f$ almost everywhere.

Finally, since each $T_s$ is an $L^p$ contraction,
\[
 \|T_sf-f\|_p
 \le
 2\|f-Q\|_p+\|T_sQ-Q\|_p.
\]
Choose $Q$ sufficiently close to $f$ in $L^p(G)$ and then let $s\to\infty$. This proves the $L^p$ convergence.
\end{proof}

If each $\lambda_s$ is symmetric, then $\check\lambda_s=\lambda_s$ and $\mu_s=\lambda_s*\lambda_s$. Thus Theorem~\ref{thm:compact-group-maximal} applies to families $(\mu_s)$ with convolution square roots $\lambda_s$ satisfying the factorization condition \eqref{eq:compact-chain}. The product Poisson measures satisfy this condition.

Let
\[
 \mathbb Z^{(\mathbb N)}
 :=\{\gamma=(\gamma_n)_{n\ge1}\in\mathbb Z^\infty:
 \gamma_n=0\text{ for all but finitely many }n\},
 \qquad
 z^\gamma:=\prod_{n\ge1}z_n^{\gamma_n}.
\]
For $r=(r_n)_{n\ge1}\in[0,1]^\infty$, let $\nu_r$ be the product probability measure whose $n$th factor is $P_{r_n}\,dm$ when $r_n<1$ and the point mass at $1$ when $r_n=1$. Equivalently, $\widehat{\nu_r}(\gamma)=\prod_{n\ge1}r_n^{|\gamma_n|}$ for $\gamma\in\mathbb Z^{(\mathbb N)}$.
For $g\in L^p(\T^\infty)$, define $P_rg:=\nu_r*g$. Then $\nu_a*\nu_b=\nu_{ab}$ and $P_aP_b=P_{ab}$.

The next theorem specializes Theorem~\ref{thm:compact-group-maximal} to componentwise increasing radial approaches.

\begin{theorem}\label{thm:coordinatewise-monotone}
Let $(r^{(s)})_{s\ge1}$ be a boundary-point-independent approach and suppose
\[
 0\le r_n^{(1)}\le r_n^{(2)}\le\cdots\uparrow1
 \qquad(n\ge1).
\]
Then, for every $1<p<\infty$,
\[
 \left\|\sup_{s\ge1}|P_{r^{(s)}}g|\right\|_p
 \le \frac{p}{p-1}\|g\|_p,
 \qquad g\in L^p(\T^\infty),
\]
and
\[
 P_{r^{(s)}}g\longrightarrow g
 \quad\text{almost everywhere and in }L^p(\T^\infty).
\]
\end{theorem}

\begin{proof}
Put $q^{(s)}=\sqrt{r^{(s)}}$ and $\lambda_s=\nu_{q^{(s)}}$.
Since $q^{(s)}\le q^{(s+1)}$ coordinatewise, define
\[
 \alpha_n^{(s)}=
 \begin{cases}
 q_n^{(s)}/q_n^{(s+1)},&q_n^{(s+1)}>0,\\
 0,&q_n^{(s)}=q_n^{(s+1)}=0.
 \end{cases}
\]
Then $\lambda_s=\lambda_{s+1}*\nu_{\alpha^{(s)}}$. Each product Poisson measure is symmetric, so $\lambda_s*\check\lambda_s=\nu_{q^{(s)}}*\nu_{q^{(s)}}=\nu_{r^{(s)}}$. Moreover, $r_n^{(s)}\to1$ for every fixed $n$ implies $\nu_{r^{(s)}}\overset{w^*}{\longrightarrow}\delta_{(1,1,\ldots)}$.
The convergence holds first on each character $z^\gamma$, since $\gamma$ has finite support, and then on $C(\T^\infty)$ by density of the trigonometric polynomials. Theorem~\ref{thm:compact-group-maximal} now gives the conclusion.
\end{proof}

\begin{corollary}\label{cor:Hp-radial-convergence}
Let $1<p\le\infty$, and put
\[
 X_p=
 \begin{cases}
 \ell^2,&1<p<\infty,\\
 c_0,&p=\infty.
 \end{cases}
\]
Let $f\in H^p(\T^\infty)$, and let $F$ denote its canonical holomorphic extension to $\D^\infty\cap X_p$. Suppose that $\{r^{(s)}\}_{s\ge1}\subset\D^\infty\cap X_p$ is boundary-point independent and satisfies
\[
 0\le r_n^{(1)}\le r_n^{(2)}\le\cdots\uparrow1
 \qquad(n\ge1).
\]
Then
\[
 F\bigl(r^{(s)}e^{i\theta}\bigr)
 \longrightarrow f(e^{i\theta})
\]
for $m_\infty$-almost every $e^{i\theta}\in\T^\infty$.
\end{corollary}

\begin{proof}
For each $s$, the canonical holomorphic extension agrees almost everywhere with the product Poisson extension: $F\bigl(r^{(s)}e^{i\theta}\bigr)=P_{r^{(s)}}f(e^{i\theta})$.
Since $s$ ranges over a countable set, these identities hold simultaneously outside a single null set. If $1<p<\infty$, the conclusion follows directly from Theorem~\ref{thm:coordinatewise-monotone}. If $p=\infty$, then $f\in L^2(\T^\infty)$, so the same theorem applied with $p=2$ gives the result.
\end{proof}

\vspace{1cm}


\section*{Appendix: Geometric illustrations of the selected radii}\label{app:figures}

The following figures illustrate the geometry of the construction in Section~2. They are included only for visualization and are not used in the proofs.

\begin{figure}[!htbp]
\centering
\includegraphics[width=0.72\textwidth]{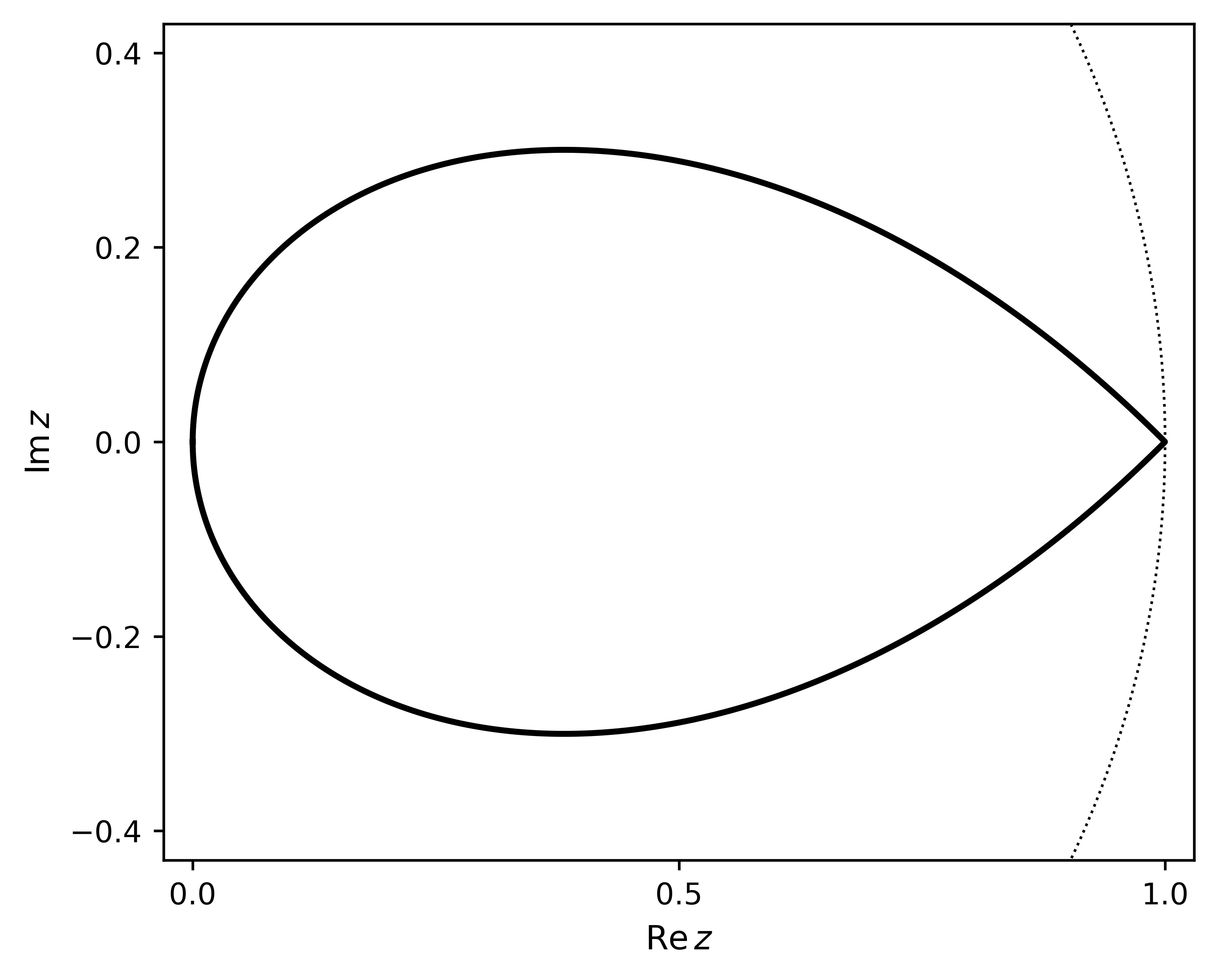}
\caption{The maximizing radii for the one-variable Poisson kernel $a\mapsto P_a(e^{i\varphi})$. The locus $z=a^*(\varphi)e^{i\varphi}$ is leaf-shaped inside the unit disc.}
\label{fig:poisson-kernel-leaf}
\end{figure}

\begin{figure}[!htbp]
\centering
\includegraphics[width=0.72\textwidth]{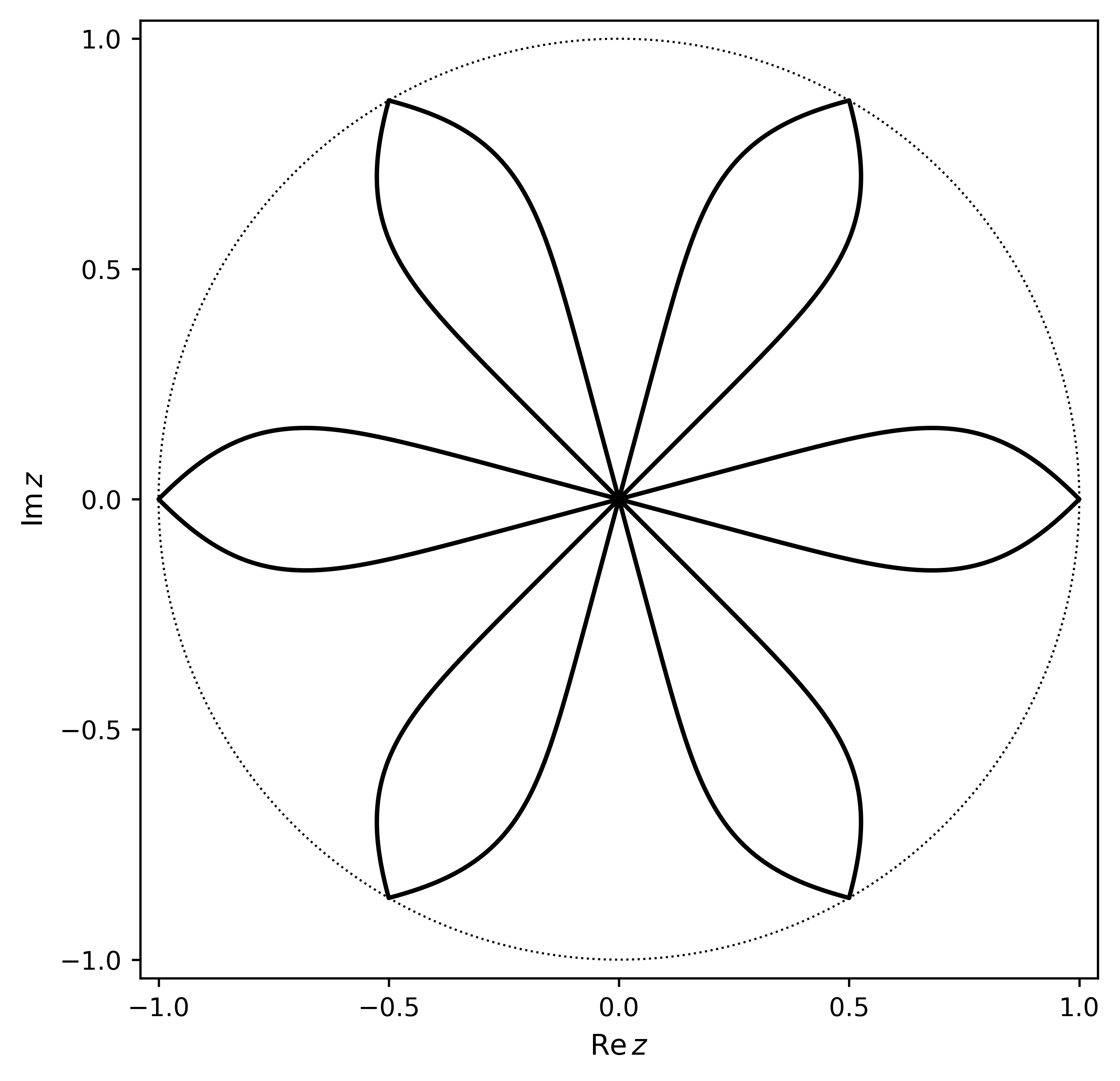}
\caption{A representative petal configuration for the maximizing radii associated with $u_{k,j}(z)=\eta_kP_{1-\eta_k}(z^{q_{k,j}})$. The displayed example uses $q_{k,j}=6$.}
\label{fig:ukj-petals}
\end{figure}

\begin{figure}[!htbp]
\centering
\includegraphics[width=\textwidth]{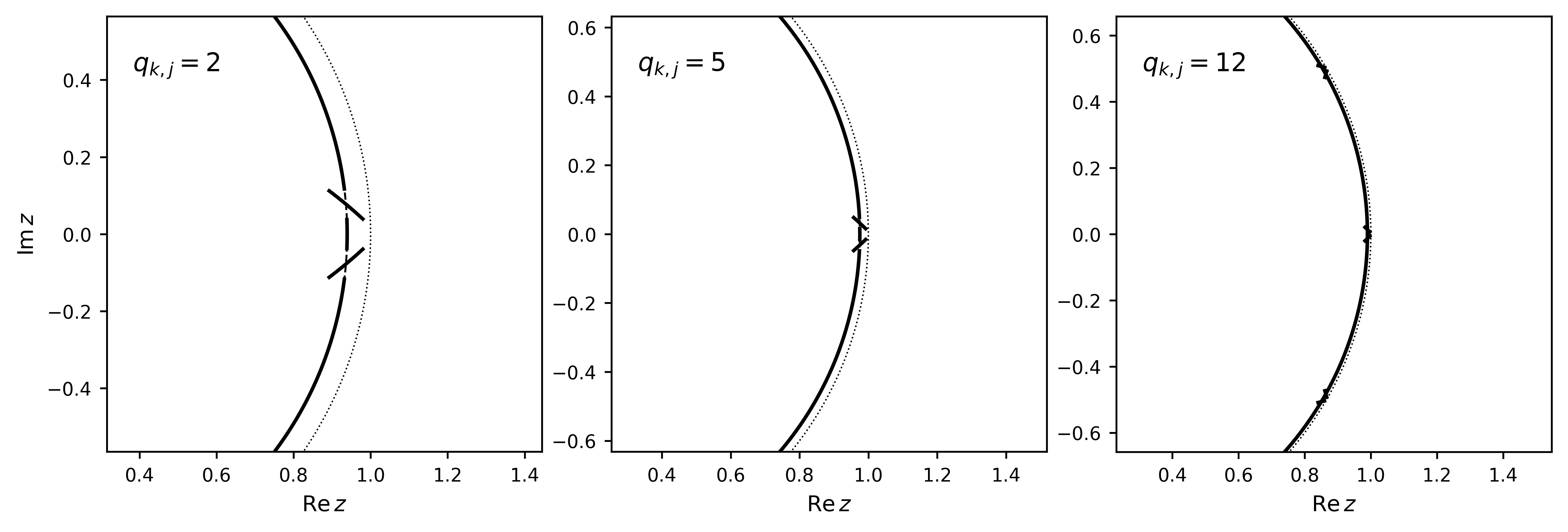}
\caption{The piecewise-defined radii $r_{k,j}(t)$ for the representative values $q_{k,j}=2,5,12$. The dashed circle is the constant radius $(\alpha_k\beta_k)^{1/(2q_{k,j})}$; the chosen radius differs from it only when $2\eta_k<\varphi_{k,j}(t)<1/4$.}
\label{fig:chosen-radii}
\end{figure}



\vspace{1cm}

\section*{Author Contribution Statement}
All authors contributed equally to the research, analysis, and preparation of the manuscript. Each author approved the final version of the paper and agrees to be accountable for all aspects of the work.

\section*{Conflict of interest}
The authors declare no conflicts of interest.

\section*{Data availability statement}
No data, models, or code were generated or used for the research described in the
article.

\section*{Acknowledgment}

Yufeng Lu was supported by the National Natural Science Foundation
of China (Grant No. 12031002). Chao Zu was supported by the National
Natural Science Foundation of China (Grant No. 12401151).

\end{document}